\documentclass[11pt]{article}
\usepackage{amssymb}
\usepackage{amsmath}
\usepackage{amsthm}
\usepackage{enumerate}
\usepackage{color}
\usepackage{xcolor}
\usepackage{fullpage}
\usepackage{hyperref}
\usepackage{comment}
\newtheorem{theorem}{Theorem}[section]

\newtheorem{corollary}[theorem]{Corollary}
\newtheorem{lemma}[theorem]{Lemma}

\newtheorem*{claim*}{Claim}

\newtheorem{remark}[theorem]{Remark}

\newtheorem{qn}[theorem]{Question}

\newtheorem{theorem LSV}{Theorem LSV}
\newtheorem*{theorem LSV*}{Theorem LSV}

\newtheorem{theorem MTP}{Mass Transference Principle}
\newtheorem*{theorem MTP*}{Mass Transference Principle}
\newtheorem{theorem K}{Khintchine's Theorem}
\newtheorem*{theorem K*}{Khintchine's Theorem}
\newtheorem{theorem J}{Jarn\'{\i}k's Theorem}
\newtheorem*{theorem J*}{Jarn\'{\i}k's Theorem}
\newtheorem{theorem KJ}{Khintchine--Jarn\'{\i}k Theorem}
\newtheorem*{theorem KJ*}{Khintchine--Jarn\'{\i}k Theorem}
\newtheorem{theorem BV1}{Theorem BV1}
\newtheorem*{theorem BV1*}{Theorem BV1}
\newtheorem{theorem BV2}{Theorem BV2}
\newtheorem*{theorem BV2*}{Theorem BV2}
\newtheorem{theorem KG}{Theorem KG}
\newtheorem*{theorem KG*}{Theorem KG}
\newtheorem{theorem IHKG}{Inhomogeneous Khintchine--Groshev Theorem}
\newtheorem*{theorem IHKG*}{Inhomogeneous Khintchine--Groshev Theorem}
\newtheorem{theorem DLN1}{Theorem DLN1}
\newtheorem*{theorem DLN1*}{Theorem DLN1}
\newtheorem{theorem DLN2}{Theorem DLN2}
\newtheorem*{theorem DLN2*}{Theorem DLN2}
\newtheorem{theorem DLN3}{Theorem DLN3}
\newtheorem*{theorem DLN3*}{Theorem DLN3}
\newtheorem{theorem S}{Theorem S}
\newtheorem*{theorem S*}{Theorem S}
\newtheorem*{theorem L*}{Previous results for homogeneous metric Diophantine approximation}
\newtheorem*{theorem LL*}{Previous results for inhomogeneous metric Diophantine approximation}
\newtheorem*{theorem LLL*}{Previous results for (inhomogeneous) multiplicative metric Diophantine approximation}

\numberwithin{equation}{section}

\renewcommand{\le}{\leq}
\renewcommand{\ge}{\geq}

\def\le{\leqslant} \def\ge{\geqslant}

\def \leq {\le}
\def \geq {\ge}

\newcommand{\Addresses}{{% additional braces for segregating \footnotesize
		\bigskip
		\footnotesize
		
		\medskip
		
		H.~Yu, 
		\textsc{Department of Pure Mathematics and Mathematical Statistics, Centre for Mathematical Sciences, Cambridge, CB3 0WB, UK}\par\nopagebreak
		\textit{E-mail address:} \texttt{hy351@maths.cam.ac.uk}
}}

\title{On the metric theory of inhomogeneous Diophantine approximation: An Erd\H{o}s-Vaaler type result}
\author{ Han Yu \footnote{Supported by the European Research Council (ERC) under the European Union’s Horizon 2020 research and innovation programme (grant agreement No. 803711), and indirectly by Corpus Christi College, Cambridge}}

\date{}

\begin{document}
	\maketitle
	\begin{abstract}
		In 1958, Sz\"{u}sz proved an inhomogeneous version of Khintchine's theorem on Diophantine approximation.  Sz\"{u}sz's theorem states that for any non-increasing approximation function $\psi:\mathbb{N}\to (0,1/2)$ with $\sum_q  \psi(q)=\infty$ and any number $\gamma,$ the following set
		\[
		W(\psi,\gamma)=\{x\in [0,1]: |qx-p-\gamma|< \psi(q) \text{ for infinitely many } q,p\in\mathbb{N}\}
		\]
		has full Lebesgue measure. Since then, there are very few results in relaxing the monotonicity condition. In this paper, we show that if $\gamma$ is can not be approximate by rational numbers too well, then the monotonicity condition can be replaced by the upper bound condition $\psi(q)=O((q(\log\log q)^2)^{-1}).$ In particular, this covers the case when $\gamma$ is not Liouville, for example $\pi,e,\ln 2, \sqrt{2}.$ In general, if $\gamma$ is irrational, $\psi(q)=O(q^{-1}(\log\log q)^{-2})$ and in addition,
		\[
		\left(\liminf_{Q\to\infty} \sum_{q=Q}^{Q^{(\log Q)^{1/8} }}\psi(q)\right)=\infty,
		\]
		then $W(\psi,\gamma)$ has full Lebesgue measure. Our proof is based on a quantitative study of the discrepancy for irrational rotations. 
	\end{abstract}
	\noindent{\small 2010 {\it Mathematics Subject Classification}\/: Primary:11J83,11J20,11K60.}
	%}\bigskip
	
	\noindent{\small{\it Keywords and phrases}\/: Inhomogeneous Diophantine approximation, Metric number theory}
	
	\section{Introduction}
	\subsection{Background}
	In this paper, we study inhomogeneous metric Diophantine approximation. Let $\gamma\in\mathbb{R}$ and $\psi:\mathbb{N}\to \mathbb{R}^+$ be a function (approximation function). We are interested in the following set:
	\[
	W(\psi,\gamma)=\{x\in [0,1]: \|qx-\gamma\|< \psi(q) \text{ for infinitely many } q\in\mathbb{N}\}.
	\]
	In the case when $\gamma=0,$ the study of $W(\psi,0)$ is referred to as classical (or homogeneous) Diophantine approximation. Now, we have a complete understanding of the Lebesgue measure of $W(\psi,0).$
	
	\begin{theorem L*}\footnote{This list of results is by no means complete.}
		The set $W(\psi,0)$ has full Lebesgue measure if:
		\begin{itemize}
			\item {Khintchine's theorem \cite{Khintchine}}: $\psi$ is non-increasing and $\sum_q \psi(q)=\infty.$
			\item {Duffin-Schaeffer's theorem \cite{DS}}: $\sum_q \psi(q)=\infty$ and
			\[
			\limsup_{Q\to\infty} \frac{\sum_{q\leq Q} \psi(q)\phi(q)/q}{\sum_{q\leq Q} \psi(q)}>0,
			\]
			where $\phi(.)$ is the Euler Phi function.
			\item {Erd\H{o}s-Vaaler's theorem \cite{Vaaler}}: $\sum_q \psi(q)\phi(q)/q=\infty$ and $\psi(q)=O(q^{-1}).$ 
			\item {Extra divergence \cite{BHHV}, \cite{ALMTZ}}: $\sum_q \frac{\psi(q)\phi(q)}{q \log^C q}=\infty$ for a number $C>0.$
			\item {Final result \cite{KM2019}}: $\sum_q \psi(q)\phi(q)/q=\infty.$ This result settled the challenging Duffin-Schaeffer conjecture posed in \cite{DS}. This result also has an easy-to-prove convergence part: if $\sum_q \psi(q)\phi(q)/q<\infty$ then for Lebesgue almost all $x\in [0,1]$ there are at most finitely many coprime pairs $(p,q)$ such that $|qx-p|<\psi(q).$
		\end{itemize}
	\end{theorem L*}
	If $\gamma\neq 0,$ the study of $W(\psi,\gamma)$ is referred to as inhomogeneous metric Diophantine approximation. Much less is known for the inhomogeneous case. We list some of them here.
	\begin{theorem LL*}\footnote{This list is not complete.}
		\begin{itemize}
			\item {Sz\"{u}sz's theorem}: If $\psi$ is non-increasing and $\sum_q \psi(q)=\infty$ then $W(\psi,\gamma)$ has full Lebesgue measure for all real numbers $\gamma.$
			\item {Ram\'{i}rez's examples}: Without the monotonicity of the approximation function $\psi$, the condition $\sum_q \psi(q)=\infty$ alone cannot ensure that $W(\psi,\gamma)$ has positive Lebesgue measure. 
			\item {Extra divergence \cite{Yu}}: For each $\epsilon>0,$ if $\sum_{q} q^{-\epsilon} \psi(q)=\infty,$ then for all numbers $\gamma,$ $W(\psi, \gamma)$ has full Lebesgue measure.
		\end{itemize}
	\end{theorem LL*}
	We can compare the above list to that of the homogeneous case.  Sz\"{u}sz's theorem takes the place of  Khintchine's theorem. The extra divergence result in \cite{Yu} is a rather weaker companion of the extra divergence result in \cite{BHHV}, \cite{ALMTZ}. Other than those, there are no further inhomogeneous analogues of the homogeneous results. In fact, the new technical difficulty for studying inhomogeneous Diophantine approximation can be partially seen when one attempts to directly apply Duffin-Schaeffer's argument in \cite{DS} or Vaaler's argument in \cite{Vaaler} to the inhomogeneous case.  We shall leave those fun tasks to the reader (hint: an interval of length less than one can still contain an integer).  Thus, the inhomogeneous shift $\gamma$ really creates something substantially new. The ultimate goal for the inhomogeneous case is to obtain a result as in \cite{KM2019}. The idea of the arguments in \cite{KM2019} partially feature the idea in \cite{Vaaler}. In this way, one might think that an inhomogeneous version of Erd\H{o}s-Vaaler's theorem can provide us with some new lights. This is in fact the main motivation for writing this paper.
	
	One application of  inhomogeneous Diophantine approximation is to understand multiplicative Diophantine approximation. Let $\psi$ be an approximation function and let $\gamma_1,\gamma_2$ be real numbers. We consider the set
	\[
	W(\psi,\gamma_1,\gamma_2)=\{(x,y)\in [0,1]^2: \|qx-\gamma_1\|\|qy-\gamma_2\|<\psi(q)\text{ infinitely often}\}.
	\]
	A famous conjecture of Littlewood states that for an arbitrary $\epsilon>0,$ consider the approximation function $\psi(q)=\epsilon/q,$ the set $W(\psi,0,0)$ contains all pairs of irrational numbers $(x,y)\in [0,1]^2.$ On the metric aspect, what is interesting is to find sufficient conditions on numbers $\gamma_1,\beta,\gamma_2$ such that $W(\psi,\gamma_1,\gamma_2)\cap \{y=\beta\}$ has full Lebesgue measure. For convenience, we consider the following set,
	\[
	W(\psi,\beta,\gamma_1,\gamma_2)=\{x\in [0,1]: \|qx-\gamma_1\|\|q\beta-\gamma_2\|<\psi(q)\text{ infinitely often}\}.
	\]
	Recently, there have been many significant progresses in understanding the Lebesgue measure of $W(\psi,\beta,0,\gamma_2).$
	\begin{theorem}[\cite{BHV},\cite{KM2019},\cite{C18}]
		Let $\psi$ be a monotonic approximation function. If $\sum_q \psi(q)\log q=\infty,$ then for each irrational non-Liouville number $\beta,$ real number $\gamma_2,$ the set $W(\psi,\beta,0,\gamma_2)$ has full Lebesgue measure.
	\end{theorem}
	
	In \cite{BHV} it was proved that the above result follows from the Duffin-Schaeffer conjecture which was later proved in \cite{KM2019}. However, before the appearance of \cite{KM2019}, Chow \cite{C18} proved the above result without relying on the Duffin-Schaeffer conjecture.
	
	\subsection{Results in this paper}
	We will focus on relaxing the monotonicity condition in Sz\"{u}sz's theorem. By the results in \cite{Ramirez}, the monotonicity condition cannot be dropped without introducing other conditions. As we mentioned in the previous section, the ultimate goal in this direction is to find a necessary and sufficient condition on $\psi,\gamma$ for $W(\psi,\gamma)$ to have full Lebesgue measure as in the Duffin-Schaeffer conjecture. This is currently beyond our reach. At this stage, we pose the following question which could be challenging.\footnote{The work in \cite{KM2019} certainly also sheds some lights on the inhomogeneous problem but does not make it any less challenging.}
	\begin{qn}
		Prove or disprove: Let $\psi$ be an approximation function and $\sum_{q} \psi(q)\phi(q)/q=\infty.$ Then for each real number $\gamma,$ the Lebesgue measure of $W(\psi,\gamma)$ is one.
	\end{qn}
	
	Towards this direction, we will provide the following Erd\H{o}s-Vaaler type result for inhomogeneous Diophantine approximation. The notion of tamely/wildly Liouville numbers will be introduced in Section \ref{pre}. To have some ideas, all non-Liouville numbers are tamely Liouville. There are tamely Liouville numbers which are Liouville.
	
	\begin{theorem}[Main theorem I]\label{Main1}
		Let $\psi(q)=O((q\log\log^2 q)^{-1})$ and $\sum_{q} \psi(q)=\infty.$ Then for each tamely Liouville number $\gamma,$ the Lebesgue measure of $W(\psi,\gamma)$ is one.
	\end{theorem}
	The above theorem holds for example when $\gamma=e,\pi$, an algebraic irrational number or the natural logarithm of a non-zero algebraic number. To say that $\gamma$ is tamely Liouville is essentially syaing that
	\[
	\|q\gamma\|\gg q^{-(\log q)^{1/4}}.
	\]
	 We provide more precise details in Section \ref{pre}. Here the tamely Liouville condition can be weakened if we have more information of the support of $\psi.$ For example, if $\psi(q)\neq 0$ only for $q$ with $d(q)=O(\log q)$ then we only need to require that
	 \[
	 \|q\gamma\|\gg q^{-(\log q)^{1/2}}.
	 \]
	\begin{comment}
	We can add some extra conditions to ensure that $W(\psi,\gamma)$ to have full Lebesgue measure. 
	\begin{theorem}\label{add}
	Under the condition in Theorem \ref{Main1}, suppose further that there are infinitely many integers $M>1$ with
	\[
	\sum_{q\geq 1} \psi(qM)=\infty.
	\]
	Then $W(\psi,\gamma)$ has full Lebesgue measure.
	\end{theorem}
	\end{comment}
	An application of Theorem \ref{Main1} yields the following. 
	\begin{corollary}\label{MUL1}
		Let $\gamma_1$ be an irrational number which is tamely  Liouville. Let $\gamma_2$ be a real number and $\beta$ be an irrational number. Suppose that $\psi(q)=O((q\log q(\log\log q)^2)^{-1})$ and
		\[
		\sum_{q: \|q\beta-\gamma_2\|\geq (\log q)^{-1}} \frac{\psi(q)}{\|q\beta-\gamma_2\|}=\infty.\tag{D}
		\]  Then the set $W(\psi,\beta,\gamma_1,\gamma_2)$ has Lebesgue measure is one.
		
	\end{corollary}
	
	The divergence condition (D) is sometimes not easy to check as it also depends on the choice of $\beta$ and $\gamma_2.$ In order to be sure that this condition is possible to be satisfied, we provide an example in Section \ref{Proof}. In fact for non-Liouville $\beta$ and tamely Liouville $\gamma_1$, it is possible the check that $|W(\psi,\beta,\gamma_1,\gamma_2)|=1$ for monotonic $\psi$ with $\sum_q \psi(q)\log\log q=\infty$. We do not prove this result here as we wish to prove a stronger result in a forthcoming project.
	
	Several remarks are in order.
	\begin{comment}
	\begin{remark}
	For inhomogeneous approximation, we cannot improve the positive measure result to full measure result directly as in the homogeneous case where Cassels-Gallagher's zero-one law \cite{C50}, \cite{Gallagher}  is at hand. We do believe that $4^{-1}$ can be improved to be $1$ without introducing any other conditions.
	\end{remark}
	\end{comment}
	\begin{remark}
		We need $\psi(q)=O((q\log\log^2 q)^{-1})$ which is a stronger condition than the Erd\H{o}s-Vaaler condition that $\psi(q)=O(1/q)$. We do believe that $\psi(q)=O(q^{-1})$ would be sufficient in Theorem \ref{Main1}. 
	\end{remark}
	The requirement that $\gamma$ needs to be tamely Liouville comes from the method we will be using. In fact, we need this condition for making  Estimate (III) in the proof of Lemma \ref{Sum Estimate} valid. It is possible to pose weaker condition on $\gamma$ in terms of its Diophantine property. We cannot completely remove this condition yet. In this direction, we supply the following result to deal with the case when $\gamma$ is wildly Liouville.
	
	In what follows, we call an approximation function $\psi$ to be \emph{weakly extra divergent} if
	\begin{align*}
	\left(\liminf_{Q\to\infty} \sum_{q=Q}^{Q^{(\log Q)^{1/8}}} \psi(q)\right)=\infty.\tag{WEx}
	\end{align*}
	It is possible to check that $\psi(q)=1/(q\log q)$ is weakly extra divergent.
	\begin{theorem}[Main theorem II]\label{Main2}
		Let $\psi$ be weakly extra divergent and $\psi(q)=O((q\log\log^2 q)^{-1}).$  Then for each wildly Liouville number $\gamma,$ the Lebesgue measure of $W(\psi,\gamma)$ is one.
	\end{theorem}
	From here, one might be wondering what is the strength of the above result compared with Sz\"{u}sz's theorem. Without too much extra efforts, we will use the method for proving Theorems \ref{Main1} to revisit Sz\"{u}sz's theorem when $\gamma$ is tamely Liouville.
	\begin{theorem}[Revisited Sz\"{u}sz's theorem for tamely Liouville inhomogeneous shift]
		Let $\psi(q)$ be a non-increasing approximation function with $\sum_q \psi(q)=\infty.$ Then for each tamely Liouville number $\gamma,$ $|W(\psi,\gamma)|=1.$
	\end{theorem}
	When $\gamma$ is wildly Liouville, the result still holds under the extra weak divergence (WEx). Some extra efforts need to be made. The central idea is to expand $\gamma$ with continued fraction and analyse more carefully the partitial convergents $P_i/Q_i,i\geq 0.$ Then for $Q\in [Q_i,Q_{i+1}),$ one can effectively regard $\gamma$ as being a rational number with denominator $Q_i$ (with an absolute error less than $1/(Q_iQ_{i+1})$). Since $\psi$ is monotonic, it is possible to study $\psi$ on a subset of blocks $[Q_i,Q_{i+1}]$. However, this essentially brings us back to the original arguments of Sz\"{u}sz (\cite{Szusz}) which used a fine analysis of the continued fraction of $\gamma$ together with a generic property of the continued fraction of $x\in [0,1].$ For this reason, we decide not to fully reprove Sz\"{u}sz's theorem. 

	Finally, for completeness, we also add the following higher dimensional analogies of Sz\"{u}sz's theorem without monotonicity requirement. The results are very likely to be known. We will provide a simple proof at the very end of this paper.
	\begin{theorem}\label{HIGH}
		Let $k\geq 1$ be an integer. Let $\gamma_1,\dots,\gamma_k$ be real numbers. Let $\psi$ be an approximation function. Let $\epsilon>0$ be a small number. We require that
		\[
		\begin{cases}
		\sum_q (\psi(q))^k=\infty & k\geq 3, \\
		\sum_q (\psi(q)\phi(q)/q)^2=\infty & k=2\\
		\sum_q \psi(q)/d(q)^{1+\epsilon}=\infty & k=1
		\end{cases}
		\]
		Then, the set of points in $[0,1]^k$ with infinitely many integers $q$ with
		\[
		\max_{i=1,\dots,k}\{\|qx_i-\gamma_i\|\}\leq \psi(q)
		\]
		has Lebesgue measure one.
	\end{theorem}
	\begin{remark}
		For $k=2,$ we need a slightly stronger divergence condition. This condition has nothing to do with the Duffin-Schaeffer conjecture although it appears to be closely related. For $k=1,$ the divergence condition here is much stronger than $\sum_q \psi(q)=\infty.$ The proof of this theorem is much simpler than the theorems above. The simple method already provides quite satisfactory results for $k\geq 2.$ This is not a surprise. For the homogeneous case, we already know Gallagher's result \cite{Gallagher65} for higher dimensional Khintchine's theorem without monotonicity and Pollington and Vaughan's result \cite{PV} for higher dimensional Duffin-Schaeffer conjecture. Pollington and Vaughan's proof in \cite{PV} is much less involved than the proof in \cite{KM2019} for the one dimensional Duffin-Schaeffer conjecture.
	\end{remark}
	\section{Notation}\label{notation}
	
	\begin{itemize}
		\item $A^{\psi,\gamma}_q$: 
		Let $\psi$ be an approximation function and $\gamma$ be a real number. For each integer $q\geq 1,$ we use $A^{\psi,\gamma}_q$ to denote the set
		\[
		A^{\psi,\gamma}_q=\{x\in [0,1]: \|qx-\gamma\|<\psi(q)\}.
		\]
		We can assume that $\psi(q)<1/2$ for all $q\geq 1.$ In fact, if there are infinitely many $q$ with $\psi(q)\geq 1/2,$ then $W(\psi,\gamma)$ would be the whole unit interval. If $\gamma,\psi$ are clear from the context, we will write $A_q$ instead of $A^{\psi,\gamma}_q.$
		\item $\chi_A$: The indicator function of a set $A.$
		\item $B(x,r)$: Metric ball centred at $x$ with radius $r,$ where $r>0$ and $x$ belongs to a metric space.
		\item $\Delta_{\psi}(q,q')$: The value $q\psi(q')+q'\psi(q)$, where $\psi$ is a given approximation function and $q,q'$ are positive integers. When $\psi$ is clear from the context, we write it as $\Delta(q,q').$
		\item $\|x\|$: The distance of a real number $x$ to the set of integers.
		\item $\{x\}$: The unique number in $(-1/2,1/2]$ with $\{x\}-x$ being an integer.
		\item $\log $: Base $2$ logarithmic function. 
		\item $\mathcal{I}_M$: The collection of intervals $I\subset [0,1]$ of length $1/M$ and with endpoints in $M^{-1}\mathbb{N},$ where $M\geq 1$ is an integer.
		\item $|A|$: The Lebesgue measure of $A\subset \mathbb{R}$ where $A$ is a Lebesgue measurable set.
		\item Natural densities: Let $A\subset\mathbb{N}.$ The upper natural density of $A$ is 
		\[
		\limsup_{q\to\infty} \frac{\#A\cap [1,q]}{q}.
		\] 
		The lower natural density of $A$ is
		\[	
		\liminf_{q\to\infty} \frac{\#A\cap [1,q]}{q}.
		\]
		\item Asymptotic symbols: For two functions $f,g:\mathbb{N}\to (0,\infty)$ we use $f=O(g)$ to mean that there is a constant $C>0$ with
		\[
		f(q)\leq Cg(q)
		\]
		for all $q\geq 1.$
		We use $f=o(g)$ to mean that 
		\[
		\lim_{q\to\infty} \frac{f(q)}{g(q)}=0.
		\]
		For convenience, we also use $O(g), o(g)$ to denote an auxiliary function $f$ with the property that $f=O(g)$, $o(g)$ respectively. The precise form of the function $f$ changes across the contexts and it can always be explicitly written down.
	\end{itemize}

	\section{Preliminary}\label{pre}
	\subsection{Borel-Cantelli lemmas}
	There are several standard results that will be needed in the proofs of the main results. The first one is the Borel-Cantelli lemma. The following result can be found in \cite[Proposition 2]{BDV ref}.
	\begin{lemma}[Borel-Cantelli]\label{Borel}
		Let $(\Omega, \mathcal{A}, m)$ be a probability space and let $E_1, E_2, \ldots \in \mathcal{A}$ be a sequence of events in $\Omega$ such that $\sum_{n=1}^{\infty}{m(E_n)} = \infty$. Then 
		\[m(\limsup_{n \to \infty}{E_n}) \geq \limsup_{Q \to \infty}{\frac{\left(\sum_{s=1}^{Q}{m(E_s)}\right)^2}{\sum_{s,t=1}^{Q}{m(E_s \cap E_t)}}}.\]
		If $\sum_{n=1}^{\infty}{m(E_n)} <\infty$ then $m(\limsup_{n \to \infty}{E_n})=0.$
	\end{lemma}
	\begin{remark}
		The condition that $\sum_{n=1}^{\infty}{m(E_n)} = \infty$ is very essential in the statement. Suppose that $\{E_n\}_{n\geq 1}$ is a sequence of pairwise independent events and $\sum_n m(E_n)<\infty.$ Then we know that $m(\limsup_{n \to \infty}{E_n})=0.$ However, we nonetheless have
		\begin{align*}
		\limsup_{Q\to\infty}{\frac{\left(\sum_{s=1}^{Q}{m(E_s)}\right)^2}{\sum_{s,t=1}^{Q}{m(E_s \cap E_t)}}}&=&\limsup_{Q\to\infty}\frac{\left(\sum_{s=1}^{Q}{m(E_s)}\right)^2}{\left(\sum_{s=1}^{Q}{m(E_s)}\right)^2+\sum_{s=1}^Q (m(E_s)-m^2(E_s))}\\
		&\geq&\limsup_{Q\to\infty}\frac{\left(\sum_{s=1}^{Q}{m(E_s)}\right)^2}{\left(\sum_{s=1}^{Q}{m(E_s)}\right)^2+\sum_{s=1}^Q m(E_s)}\\
		&=&\limsup_{Q\to\infty}\frac{1}{1+\frac{1}{\sum_{s=1}^Q m(E_s)}}>0.
		\end{align*}
	\end{remark}
	In general, we have the following estimate from \cite{EC}.
	\begin{lemma}[Chung--Erd\H{o}s Inequality] \label{EC}
		Let $N\geq 1$ be an integer. Let $(\Omega, \mathcal{A}, m)$ be a probability space and let  $\{E_n\}_{1\leq n\leq N} \subset  \mathcal{A}$ be an arbitrary sequence of $m$-measurable sets in $\Omega$. Then, if $m\left(\bigcup_{n=1}^{ N}{E_n}\right)>0$, we have 
		\[m\left(\bigcup_{n=1}^{N}{E_n}\right) \geq \frac{\left(\sum_{s=1}^{N}{m(E_s)}\right)^2}{\sum_{s,t=1}^{N}{m(E_s \cap E_t)}}.\]
	\end{lemma}
\subsection{Discrepancies of irrational rotations}
	We also need the following result in discrepancy theory. See \cite[Section 1.4]{DT97} for more details. 
	\begin{lemma}\label{Discrepancy}
		For each irrational number $\alpha,$ there is a decreasig functoin $D_\alpha:\mathbb{N}\to (0,1)$ such that for all interval $I\subset (-1/2,1/2)$ we have
		\[
		\left|\frac{\#\{1\leq q\leq Q: \{q\alpha\}\in I\}}{Q}-|I|\right|\leq D_\alpha(Q).
		\]
	\end{lemma}
	For each irrational number $\alpha,$ we have $D_\alpha(Q)=o(1).$ When $\alpha$ is Badly approximable, we have $D_\alpha(Q)=O(\log Q/Q).$ If $\alpha\in \mathbb{R}\setminus\mathbb{Q}$ is not Liouville, then there is a number $\beta\in (0,1)$ with
	$
	D_\alpha(Q)=O(Q^{\beta-1}).
	$ Those bounds are special cases of \cite[Theorem 1.80]{DT97}. 	
	\subsection{Tamely/Wildly Liouville numbers}
	We will now introduce the notion of tamely Liouville numbers. Let $\alpha$ be an irrational number. For each integer $Q>0,$ we use $\sigma(Q)>0$ to be the smallest number such that
	\[
	\|q\alpha\|\geq q^{-\sigma(Q)}
	\]
	for all $q\in\{1,\dots Q\}.$ Then if $\alpha$ is not Liouville, $\sigma(.)$ is a bounded function. We say that $\alpha$ is \emph{tamely Liouville} if
	\[
	\sigma(Q)\ll (\log Q)^{1/4-\epsilon}.\tag{tamely Liouville}
	\]
	for a number $\epsilon>0.$ Thus if $\alpha$ is not Liouville, then it is tamely Liouville. Numbers which are not tamely Liouville are \emph{wildly Liouville}. If $\alpha$ is rational, then we see that it satisfies the constrain of being wildly Liouville. In some sense, wildly Liouville numbers are those which can be approximated by rational numbers in an extremely accurate way.

	\begin{comment}
	Instead of  Gallagher's $0,1$-law, we have the following more general criteria for a set to have full Lebesgue measure. It is originally in \cite[Lemma 7]{BDV ref} and we use the following version which is a direct consequence of \cite[Lemma 7]{BDV ref}.
	\begin{lemma}\label{BDV0}
	Let $E\subset [0,1]$ be a Borel set. Let $a$ be a positive number. Suppose that for each interval $I\subset [0,1]$ we have
	\[
	\frac{|I\cap E|}{|I|}\geq a.
	\]
	Then $E$ has full Lebesgue measure.
	\end{lemma}
	We deduce the following more convenient to use version.
	\begin{lemma}\label{BDV}
	Let $E\subset [0,1]$ be a Borel set. Let $a$ be a positive number. Suppose that there are infinitely many integers $M>1$ and  for each interval $I\in \mathcal{I}_M$ (intervals of length $1/N$ with endpoints in $\mathbb{Z}N^{-1}$) we have
	\[
	\frac{|I\cap E|}{|I|}\geq a.
	\]
	Then $E$ has full Lebesgue measure.
	\end{lemma}	
	\begin{proof}
	Observe that for each interval $J\subset [0,1],$ if the integer $M$ is larger than $100$ times the length of $J,$ then we can find an interval $J'\subset J$ such that $J'$ is a almost disjoint union of intervals in $\mathcal{I}_M$ and \[|J'|\geq (1-2|J|M^{-1}) |J|.\] We can then deduce that
	\[
	\frac{|J'\cap E|}{|J'|}\geq a.
	\]
	By choosing $M$ to be sufficiently large, we see that
	\[
	\frac{|J\cap E|}{|J|}\geq a/2.
	\]
	The conclusion follows then from Lemma \ref{BDV0}.
	\end{proof}
	\end{comment}
	\subsection{A standard result}
	The next result is a standard homework question, however, we have not found a proper reference and we provide a proof.
	\begin{lemma}\label{Aux}
		Let $\{a_q,q\geq 1\}$ be a non-increasing sequence of positive numbers with
		\[
		\sum_{q\geq 1} a_q=\infty.
		\] 
		Let $A\subset\mathbb{N}$ be a set with positive lower natural density. Then
		\[
		\sum_{q\in A} a_q=\infty.
		\]
	\end{lemma}
	\begin{proof}
		Since $A$ has positive lower density, we see that there are positive numbers $\epsilon,M>0$ such that 
		\[
		\#A\cap [1,Q]\geq \epsilon Q
		\]
		for all $Q\geq M.$ Without loss generality, we can assume that $\epsilon=k^{-1}$ for an integer $k>1.$ This implies that
		\[
		\#A\cap [1,kQ]\geq Q
		\]
		for each $Q\geq [M/k]+1=Q_0.$ Let $Q>Q_0.$ We consider
		\[
		\sum_{q\leq kQ, q\in A} a_q.
		\]
		There are at least $Q_0$ many elements in $A$ which are smaller than $kQ_0+1,$ the contribution of the first $Q_0$ of those numbers to the sum is at least
		\[
		Q_0 a_{kQ_0}.
		\]
		There are at least $Q_0+1$ many elements in $A$ smaller than $k(Q_0+1)+1.$ Therefore, the contribution of the first $Q_0+1$ elements is at least (remember that the first $Q_0$ elements are smaller than $kQ_0+1$)
		\[
		Q_0 a_{kQ_0}+a_{k(Q_0+1)}.
		\]
		Iterate the above argument we see that
		\begin{align*}
		\sum_{q\leq kQ, q\in A} a_q&\geq Q_0a_{kQ_0}+a_{k(Q_0+1)}+a_{k(Q_0+2)}+\dots+a_{kQ}&\\
		&\geq \frac{1}{k}\sum_{q=kQ_0}^{kQ}a_q.&
		\end{align*}
		This implies that
		\[
		\sum_{q\in A} a_q=\infty.
		\]
	\end{proof}
\subsection{Modify the approximation functoin}\label{modify}
Under the condition $\psi(q)=O(q^{-1}),$ and $\sum_q \psi(q)=\infty, $we can resctrict the support a little bit. Recall the classical Hardy-Ramanujan theorem.
\begin{theorem}
	Let $\Omega(.)$ be the number of prime powers divisor function. Then we have $$\sum_{q\leq n}|\Omega(q)-\ln\ln q|^2\ll n \ln\ln n.$$
\end{theorem}
In particular, for each $\epsilon>0,$ the number of integers $q$ smaller than $n$ with $\Omega(q)\geq (\log q)^{1/2+\epsilon}$ is
\[
\ll \frac{n}{(\log n)^{1+0.5\epsilon}}.
\]
Here the change of bases from $\ln$ to $\log$ provides some multiplicative constants which can be absorbed in the $\ll$ symbol. Since $\psi(q)=O(q^{-1}),$ we see that
\[
\sum_{q: \Omega(q)\geq (\log q)^{1/2+\epsilon}}\psi(q)<\infty.
\]
We can simply redefine $\psi(q)=0$ if $\Omega(q)\geq (\log q)^{1/2+\epsilon}.$ This will not affect the divergence of $\sum_q \psi(q).$ Thus, unless otherwise mentioned, we always assume that $\psi$ is supported on numbers $q$ with $\Omega(q)\leq (\log q)^{1/2+\epsilon}.$ Observe that
\[
d(q)\leq 2^{\Omega(q)}.
\]
Therefore, we have, on the support of $\psi,$
\[
\log d(q)\leq (\log q)^{1/2+\epsilon}.
\]
\section{Proofs of Theorems \ref{Main1},\ref{Main2}}\label{Proof}
	\subsection{A general bound for intersections}
	Let $\psi$ be an approximation function and $\gamma$ be a real number. In order to use Lemma \ref{Borel}, we need to estimate the size of intersections $A_q\cap A_{q'}.$
	\begin{lemma}\label{master}
		Let $H>2$ be an integer. Let $\psi$ be an approximation function and $\gamma$ be an irrational number. For integers $1\leq q'<q$ such that $\Delta(q',q)<H\gcd(q,q')$ we have the following estimate
		\[
		|A_q\cap A_{q'}|\leq 2(2H+1)\min\{\psi(q)/q, \psi(q')/q'\} \gcd(q,q') \chi_{B(0,\Delta(q,q')/\gcd(q,q'))}(\{\gamma(q'-q)/\gcd(q,q')\}).
		\]
		Otherwise if $\Delta(q,q')\geq H\gcd(q,q')$, we have
		\[
		|A_q\cap A_{q'}|\leq 4\psi(q)\psi(q')\left(1+ \frac{C_0}{2H}  \right),
		\]
		where $C_0>1$ is an absolute constant.
	\end{lemma}
	
	\begin{proof}
		First, we want to count the number of integer solutions ($n,n'$) to the following inequality:
		\[
		\left|\frac{n}{q}-\frac{\gamma}{q}-\frac{n'}{q'}+\frac{\gamma}{q'} \right|\leq \frac{\psi(q)}{q}+\frac{\psi(q')}{q'}
		\]
		with the restriction that
		\[
		\frac{\gamma+n}{q}\in [0,1],\frac{\gamma+n'}{q'}\in [0,1].
		\]
		We multiply $qq'/\gcd(q,q')$ to the above inequality and obtain
		\[
		|nq'/\gcd(q,q')-n' q/\gcd(q,q')-\gamma (q'-q)/\gcd(q,q')|\leq \Delta(q,q')/\gcd(q,q').
		\]
		Let $s=q/\gcd(q,q'), s'=q'/\gcd(q,q').$ We see that $\gcd(s,s')=1.$ We have $\{\gamma(q'-q)/\gcd(q,q')\}\in (-1/2,1/2).$ Suppose that $\Delta(q',q)/\gcd(q,q')<H.$ Then there are at most $2H+1$ possible integer values for $ns'-n's$. The solutions exist only when\footnote{If $\Delta(q',q)/\gcd(q,q')\geq 1,$ then the above holds trivially. Thus, this condition is only effective when $\Delta(q',q)$ is much smaller compare to $\gcd(q,q').$}
		\[
		\{\gamma(q'-q)/\gcd(q,q')\}\in B(0,\Delta(q',q)/\gcd(q,q')).
		\] 
		There are $\gcd(q,q')$ many pairs $n,n'$ with $ns'-n's$ taking each of the above values. Therefore we see that $A_q\cap A_{q'}$ is contained in the union of at most $(2H+1)\gcd(q,q')$ many intervals of length $2\min\{\psi(q)/q,\psi(q')/q'\}.$ Thus we have
		\[
		|A_q\cap A_{q'}|\leq 2(2H+1)\min\{\psi(q)/q, \psi(q')/q'\} \gcd(q,q') \chi_{B(0,\Delta(q,q')/\gcd(q,q'))}(\{\gamma(q'-q)/\gcd(q,q')\}).
		\]
		Suppose that $\Delta(q',q)/\gcd(q,q')\geq H.$ In this case, we use the Formula (3.2.5) in \cite{Harman} which says that
		\[
		||A_q\cap A_{q'}|-4\psi(q)\psi(q')|\leq C_0 \gcd(q,q')\min\{\psi(q)/q,\psi(q')/q'\}.
		\]
		Here the constant $C_0>1$ is absolute. From here we see that
		\[
		|A_q\cap A_{q'}|\leq 4\psi(q)\psi(q')\left(1+\frac{C_0}{H}\Delta(q,q')\min\{\psi(q)/q,\psi(q')/q'\}\frac{1}{4\psi(q)\psi(q')}  \right).
		\]
		Notice that
		\[
		\Delta(q,q')\leq 2qq'\max\{\psi(q)/q,\psi(q')/q'\}.
		\]
		Therefore we have
		\[
		\Delta(q,q')\min\{\psi(q)/q,\psi(q')/q'\}\leq 2\psi(q)\psi(q').
		\]
		Thus, we see that
		\[
		|A_q\cap A_{q'}|\leq 4\psi(q)\psi(q')\left(1+ \frac{C_0}{2H}  \right).
		\]
		This proves the result.
	\end{proof}
	
	From the above result we see that for each $q\geq 1$ we have
	\begin{eqnarray*}
		& &\sum_{1\leq q'<q} |A_q\cap A_{q'}|\\&\leq& 2(2H+1)\sum_{q': \Delta(q,q')<H\gcd(q,q')} \frac{\psi(q)}{q}\gcd(q',q)\chi_{B(0,\Delta(q,q')/\gcd(q,q'))}(\{\gamma (q'-q)/\gcd(q,q')\})\\&+&4(1+C_0/(2H))\sum_{1\leq q'\leq q} \psi(q)\psi(q').
	\end{eqnarray*}
	We now want to estimate the first sum on the RHS in above. We will prove several lemmas which eventually cover all the situations we will meet  later. First, we consider the  case when $\gamma$ is not 'too' Liouville.
	
	\subsection{tamely Liouville numbers}
	Let $\gamma$ be an irrational number. For each integer $Q>0,$ we use $\sigma(Q)>0$ to be the smallest number such that
	\[
	\|q\gamma\|\geq q^{-\sigma(Q)}
	\]
	for all $q\in\{1,\dots Q\}.$ We recall that $\gamma$ is \emph{tamely Liouville} if
	\[
	\sigma(Q)\ll (\log Q)^{1/4-\epsilon}.\tag{tamely Liouville}
	\]
	for a number $\epsilon>0.$
	In particular, recall that the order of the divisor function on the support of $\psi$ (recall Section \ref{modify}) $\log d(.)$ is
	\[
	\log d(q)=O((\log q)^{1/2+\epsilon}).
	\]
	We see that 
	\[
	Q^{1/\sigma^2(Q)}\gg \log Q \max_{q\in\{1,\dots,Q\},\psi(q)\neq 0}d(q).
	\]
	
	We want to closely look at the discrepancy property of the irrational rotation $\{n\gamma\}_{n\geq 1}.$  First, we recall the following result of Erd\H{o}s-Tur\'{a}n-Koksma, see \cite[Theorem 1.21]{DT97}.
	\begin{theorem}[ETK]\label{ETK}
		Let $\alpha$ be an irrational number. Let $N,H$ be positive integers. The the discrepancy of the sequence $\{n\alpha\}_{1\leq n\leq N}$ can be bounded from as follows
		\[
		D_\alpha(N)\leq 3 \left(\frac{1}{H}+2\sum_{0<h\leq H}\frac{1}{Nh}\frac{2}{\|h\alpha\|}\right).
		\]
	\end{theorem}
	Then we see that for each $N\leq Q$ and $H>0,$
	\[
	D_\alpha(N)\leq 3 \left(\frac{1}{H}+2\sum_{0<h\leq H}\frac{1}{Nh}\frac{2}{\|h\gamma\|}\right).
	\]
	We also have $\|h\gamma\|\geq h^{-\sigma(Q)}.$ From here, we see that
	\[
	\sum_{0<h\leq H}\frac{1}{Nh}\frac{2}{\|h\gamma\|}\leq\frac{2\sigma(Q)}{N}H^{\sigma(Q)}.
	\]
	We can choose $H\approx \sigma(Q)^{-1/(\sigma(Q)+1)}N^{1/(1+\sigma(Q))}$ to minimize the above sum. As a result we see that
	\[
	N D_\gamma(N)\leq 36  \sigma(Q)^{1/(1+\sigma(Q))}N^{\sigma(Q)/(1+\sigma(Q))}
	\] 
	for all $N\leq Q.$ It is possible to see that $\sigma(Q)^{1/(\sigma(Q)+1)}$ is bounded by two. Therefore, for all integers $Q\geq 1$ and all interval $I\subset [0,1],$
	\begin{align*}
	&\left|\#\{1\leq q\leq Q: \{Q\gamma\}\in I\}-Q|I|\right|\leq 72  Q^{\sigma(Q)/(1+\sigma(Q))}.&
	\end{align*}
	Aster this preparation, we can prove the following lemma.
	\begin{lemma}\label{Sum Estimate}[tamely Liouville counting lemma]
		Under the hypothesis of Lemma \ref{master}, suppose further that $\gamma$ is tamely Liouville and $\psi(q)\leq C q^{-1}(\log\log q)^{-2} $ for a number $C>0.$ Then there are constants $C',C''>0$ such that for all $q\geq 16,$
		\[
		\sum_{1\leq q'<q} \frac{\psi(q)}{q}\gcd(q',q)\chi_{B(0,\Delta(q,q')/\gcd(q,q'))}(\{\gamma (q'-q)/\gcd(q,q')\})\leq C' \frac{\psi(q)}{(\log\log q)^{2}}\sum_{r|q} \frac{\log r}{r}+C''\psi(q).
		\]
	\end{lemma}
	\begin{remark}\label{Remark}
		The $(\log\log q)^{-2}$ factor on the RHS comes from the upper bound condition on $\psi.$ Actually, exactly the same arguments would show that if $\psi(q)=O(q^{-1})$ then we have
		\[
		\sum_{1\leq q'<q} \frac{\psi(q)}{q}\gcd(q',q)\chi_{B(0,\Delta(q,q')/\gcd(q,q'))}(\{\gamma (q'-q)/\gcd(q,q')\})=O\left( \psi(q)\sum_{r|q} \frac{\log r}{r}+\psi(q)\right).
		\]
		The extra $(\log\log q)^{-2}$ factor will be important later in the proof of Theorem \ref{Main1}.
	\end{remark}
	\begin{proof}
		We can simply assume that $\psi(q)\leq (q(\log\log q)^2)^{-1}$ for all $q\geq 3$ (where $\log\log q$ is positive) and $\psi(q)=0$ for $q=1,2.$ The multiplicative constant $C$ in the statement will not affect arguments in this proof at all. 
		
		First, we observe that the sum on the LHS can be rewritten as
		\[
		A=\frac{\psi({q})}{q}\sum_{r|q} r \sum_{q': \gcd(q',q)=r}\chi_{B(0,\Delta(q,q')/r)}(\{\gamma (q'-q)/r\}).
		\]
		For each integer $k\geq 0$ we use $D_{k,r}$ to denote the set
		\[
		D_{k,r}=\{1\leq q'<q: \gcd(q',q)=r, q'/q\in [2^{-k-1},2^{-k})\}.
		\]
		Then we see that
		\[
		A=\frac{\psi(q)}{q}\sum_{r|q} r\sum_{k\geq 0} \sum_{q'\in D_{k,r}} \chi_{B(0,\Delta(q,q')/r)}(\{\gamma (q'-q)/r\}).
		\]
		The infinite sum over the index $k$ is actually a finite sum since for large $k$ the set $D_{k,r}$ would be empty. Furthermore, we have the following trivial bound
		\begin{align*}
		\sum_{1\leq q'<q^{1/2}} \frac{\psi(q)}{q}\gcd(q',q)\chi_{B(0,\Delta(q,q')/\gcd(q,q'))}(\{\gamma (q'-q)/\gcd(q,q')\})\leq \frac{\psi(q)}{q}\sum_{1\leq q'\leq q^{1/2}} \gcd(q',q)\\
		=\frac{\psi(q)}{q}\sum_{r|q} r\sum_{q':r| q', 1\leq q'\leq q^{1/2}} 1
		\leq\frac{\psi(q)}{q}\sum_{r|q} r \frac{q^{1/2}}{r}
		=\psi(q)d(q) q^{-1/2}=\psi(q)o(1).\tag{I}
		\end{align*}
		Thus we only need to consider $k$ such that $2^k\leq q^{1/2}$ since otherwise $q'\leq q/2^k\leq q^{1/2}$ and its contribution to the sum is included in Estimate (I). In fact, the number $1/2$ is of no significance. One can replace it by any positive number strictly smaller than $1.$ We just fix the value for convenience. Now we split the sum on $k$ into two parts
		\[
		\sum^{k\leq 0.5\log q}_{k\geq 0}=\sum^{k\leq 0.5\log q}_{k\leq 2\log r}+\sum^{k\leq 0.5\log q}_{k>2\log r}.
		\]
		It can happen that the second sum in above is 0. In general, we bound the second term from above as follows
		\begin{align*}
		\frac{\psi(q)}{q}\sum_{r|q} r\sum_{k> 2\log r}\sum_{q'\in D_{k,r}}\chi_{B(0,\Delta(q,q')/r)}(\{\gamma (q'-q)/r\})\\
		\leq \frac{\psi(q)}{q}\sum_{r|q} r\sum_{q': r|q', 1\leq q'\leq q/r^2} 1\leq \frac{\psi(q)}{q}\sum_{r|q} r \frac{q}{r^{3}}\leq \zeta(2)\psi(q).\tag{II}
		\end{align*}
		In what follows, we always have $k\leq 0.5\log q$ and we do not explicitly write it down. We now estimate the $\sum_{k\leq 2\log r}$ term.

		First, we have for $q'\in D_{k,r}$
		\[
		\Delta(q',q)\leq 2\frac{2^{k+1}}{\log\log^2 (q/2^{k+1})}.
		\]
		Therefore, we see that
		\[
		\sum_{q'\in D_{k,r}}\chi_{B(0,\Delta(q,q')/r)}(\{\gamma (q'-q)/r\})\leq \sum_{q'\in D_{k,r}} \chi_{B(0,2^{k+2}/(r\log\log^2 (q/2^{k+1})))}(\{\gamma (q'-q)/r\}).
		\]
		Observe that
		\[
		D_{k,r}\subset \{q': r|q', q'/q\in [2^{-k-1},2^{-k})\}=\{rs: s\geq 1, rs/q\in [2^{-k-1},2^{-k})\}.
		\]
		We need to count the number $S_{k,r}$ of $s\geq 1, rs/q\in [2^{-k-1},2^{-k})$ such that $\{\gamma(s-qr^{-1})\}$ is contained in
		\[
		I_{k,r}=B(0,2^{k+2}/(r\log\log^2 (q/2^{k+1}))).
		\]
		We remark that the requirement that 
		\[
		\frac{q}{2^{k+1} r}\leq s< \frac{q}{2^k r}
		\]
		could lead to null choice of $s.$ We treat the case when $q/(2^{k}r)>1.$ Otherwise, there is nothing to consider. In this case, we see that there are at most $1+q/(2^{k+1}r)\leq 3q/(2^{k+1}r)$ many integers $s$ in the above range. 
		
		Since $\gamma$ is tamely Liouville, we can find the corresponding $\sigma$ function for $\gamma$ which grows slowly. As we have seen before, for any set $\mathcal{N}$ of $n\geq 2$ consecutive integers  and any interval $I\subset [-1/2,1/2],$
		\[
		|\#\{m\in \mathcal{N}: \{\gamma m\}\in I\}-|I|n|\leq 72 n^{\sigma(n)/(1+\sigma(n))}.\tag{R1}
		\]
		Moreover, we have for $q\in\{1,\dots, n\}$
		\[
		\|q\gamma\|\geq q^{-\sigma(n)}.\tag{R2}
		\]
		From here, we see that
		\begin{align*}
		S_{k,r}\leq |I_{k,r}|\frac{3q}{2^k r}+72 \left(\frac{3q}{2^k r}\right)^{\sigma(3q)/(1+\sigma(3q))}.\tag{Est}
		\end{align*}
		More precisely, the variable in the above $\sigma(.)$ should be $[3q/(2^kr)].$ However, as $\sigma(.)$ is increasing, we see that the above inequality holds. Moreover, $S_{k,r}=0$ if
		\[
		|I_{k,r}|\leq (q/r)^{-\sigma(q)}.
		\]
		To see this, observe that as $s$ ranging over $\{1,\dots,q/r-1\}$, the value of $\{\gamma(s-q/r)\}$ ranges over $\{\gamma\},\dots,\{\gamma(q/r-1)\}.$ Then we use (R2) to conclude the above claim.
		Thus we can assume that
		\[
		(q/r)^{-\sigma(q)}\leq 2\frac{2^{k+1}}{r\log\log^2 (q/2^{k+1})}.
		\]
		Recall that $2^k\leq q^{1/2}.$ Therefore we have
		\[
		r^{\sigma(q)+1}\leq 16 q^{0.5+\sigma(q)}\frac{1}{(\log\log q)^2}.
		\]
		For large enough $q,$ we have
		\[
		r\leq q^{(0.5+\sigma(q))/(1+\sigma(q))}.
		\]
		Next, we examine the contribution of second term of (Est) to $A,$ 
		\begin{align*}
		&\frac{\psi(q)}{q}\sum_{r|q}r \sum_k \left(\frac{3q}{2^k r}\right)^{\sigma(3q)/(1+\sigma(3q))}&\\
		&\ll \psi(q)\log q \sum_{r|q} \left(\frac{r}{q}\right)^{1/(\sigma(3q)+1)}.&
		\end{align*}
		The divisor sum $\sum_{r|q}$ in above can be further reduced because we only need to consider the divisors $r$ with $r\leq q^{(0.5+\sigma(q))/(1+\sigma(q))}\leq q^{(0.5+\sigma(3q))/(1+\sigma(3q))}.$ Therefore we see that
		\begin{align*}
		&\frac{\psi(q)}{q}\sum_{r|q}r \sum_k \left(\frac{3q}{2^k r}\right)^{\sigma(3q)/(1+\sigma(3q))}&\\
		&\ll \psi(q)(\log q )d(q) q^{-0.5/(\sigma(3q)+1)^2}.&
		\end{align*}
		From our condition (the tamely Liouville condition) on the growth rate of $\sigma(.),$ we see that
		\[
		q^{0.5/(\sigma(3q)+1)^2}\gg d(q)\log q.
		\]
		Indeed, we have certainly $\psi(q)\neq 0.$ Otherwise there is nothing to consider. However, recall Section \ref{modify},if $\psi(q)\neq 0,$ then $\log d(q)\leq (\log q)^{1/2+\epsilon}.$ This implies that
		\begin{align*}
		&\frac{\psi(q)}{q}\sum_{r|q}r \sum_k \left(\frac{3q}{2^k r}\right)^{\sigma(3q)/(1+\sigma(3q))}=\psi(q)O(1).&\tag{III}
		\end{align*}
		Next, we examine the first term in (Est),
		\begin{align*}
		&\frac{\psi(q)}{q}\sum_{r|q}r \sum_k 2\frac{2^{k+1}}{r\log\log^2 (q/2^{k+1})} \frac{3q}{2^k r}&\\
		&\ll \psi(q) \frac{1}{(\log\log q)^2}\sum_{r|q} \frac{\log r}{r}.&\tag{IV}
		\end{align*}
		Collecting the estimates (I), (II), (III), (IV) and reindexing the constants, we see that for two suitable constants $C',C''>0,$
		\[
		\frac{\psi({q})}{q}\sum_{r|q} r \sum_{q': \gcd(q',q)=r}\chi_{B(0,\Delta(q,q')/r)}(\{\gamma (q'-q)/r\})\leq C'\psi(q)\frac{1}{\log\log^2 q} \sum_{r|q} \frac{\log r}{r}+C''\psi(q).
		\]
		This finishes the proof.
	\end{proof}
	\subsection{A lemma for wildly Liouville $\gamma$}
	Next, we consider the case when $\gamma$ is a tamely Liouville number. We now consider wildly Liouville numbers. Here, we note that we can consider a rational number as being wildly Liouville in what follows. As $\gamma$ is Liouville, there exists a function $\sigma:\mathbb{N}\to \mathbb{N}$ such that there are infinitely many integers $Q$ with
	\[
	\|Q\gamma\|< Q^{-\sigma(Q)},
	\]
	where we require that $\sigma(Q)$ are positive integers and $\sigma(Q)\to\infty$ as $Q\to\infty.$ Unlike what we did for tamely Liouville numbers here the exponents $\sigma(.)$ are not necessarily to be the best fit exponents. For example, it can happen that
	$\|Q\gamma\|\leq Q^{-100\sigma(Q)-100}.$ However, since $\gamma$ is wildely Liouville, we necessarily have a lower bound
	\[
	\sigma(Q)\gg (\log Q)^{1/4-\epsilon}\tag{wildly Liouville}
	\]
	for all $\epsilon>0.$ Thus in particular, $\sigma(Q)\gg (\log Q)^{1/8}.$ We can construct the infinite set
	\[
	L_\gamma=\{Q\in\mathbb{N}: \|Q\gamma\|<Q^{-\sigma(Q)}\}.
	\]
	Here the possible choices for $\sigma(Q)$ depends in general on $\gamma.$ In addition to the condition that $\sigma(Q)\to\infty$ as shall also require that
	\[
	(\max_{q\in\{1,\dots,Q^{\sigma(Q)/2}\},\psi(q)\neq 0}d(q))\log Q=o(Q^{2/\sigma(Q)}).\tag{Slow growing}
	\]
	This is certainly possible as $d(Q)\log Q=o(Q^{\epsilon})$ for all $\epsilon>0.$ Thus we need that $\sigma(Q)$ grows in a sufficiently slow manner. However, we need to be sure that this will not contradict (Wildly Liouville). The see this, we observe that (Slow growing) says that for a positive number $\epsilon$ which can be chosen to be arbitrarily small,
	\[
	\sigma(Q)\ll (\log Q)^{1-\epsilon}.
	\]
	Thus we still have some extra rooms for the choice of $\sigma(.).$ We shall fix one such a function $\sigma$ for each wildly Liouville $\gamma.$ If $\gamma$ is rational, we see that $L_\gamma$ is simply $k\mathbb{N}$ for an integer $k\geq 1.$ We will first prove the following lemma as we essentially met all the arguments. In what follows, we assume without loss of generality that $\sigma(Q)>14$ for all integers $Q.$
	\begin{lemma}\label{Liouville}[wildly Liouville counting lemma]
		Under the hypothesis of Lemma \ref{master}, suppose further that $\gamma$ is Liouville and $\psi(q)\leq C q^{-1}(\log\log q)^{-2} $ for a number $C>0.$ Then there are constants $C',C''>0$ such that 
		\[
		\sum_{1\leq q'<q} \frac{\psi(q)}{q}\gcd(q',q)\chi_{B(0,\Delta(q,q')/\gcd(q,q'))}(\{\gamma (q'-q)/\gcd(q,q')\})\leq C' \frac{\psi(q)}{(\log\log q)^{2}}\sum_{r|q} \frac{\log r}{r}+C''\psi(q)
		\]
		for $q\in [Q^7,Q^{\sigma(Q)/2}],$ where $Q\in L_\gamma.$
	\end{lemma}
	\begin{remark}\label{Remark2}
		Here, we have a similar conclusion as Remark \ref{Remark}. 
	\end{remark}
	\begin{proof}
		We assume that $Q\geq 2^{1000}.$ Since $Q\in L_\gamma,$ we see that $\gamma=v/Q+err$ where
		\[
		v\in\mathbb{N}, \gcd(v,Q)=1, |err|<Q^{-\sigma(Q)}.
		\]
		Let $q',q$ be integers smaller than $Q^{\sigma(Q)/2}.$ Then we have either
		\[
		\|\gamma (q'-q)/\gcd(q',q)\|>\frac{1}{Q}-\frac{2}{Q^{\sigma(Q)/2}}>\frac{1}{2Q}
		\] 
		or
		\[
		\|\gamma (q'-q)/\gcd(q',q)\|\leq \frac{2}{Q^{\sigma(Q)/2}}.
		\]
		The former happens precisely when $v (q'-q)/\gcd(q',q)$ is not a multiple of $Q,$ i.e. $(q'-q)/\gcd(q',q)$ is not a multiple of $Q.$ In this case, we see that
		\[
		\gamma (q'-q)/\gcd(q',q)
		\]
		is $2/Q^{\sigma(Q)/2}$-close to a rational number with denominator $Q.$
		
		Our goal is to estimate the following sum
		\[
		\sum_{1\leq q'<q} \gcd(q',q)\chi_{B(0,\Delta(q,q')/\gcd(q,q'))}(\{\gamma (q'-q)/\gcd(q,q')\}).
		\]
		Now the situation is a bit simpler than the non-Liouville case as now $\gamma$ can be effectively considered as being rational ($\approx v/Q$). By the arguments at the beginning of the proof of Lemma \ref{Sum Estimate}, which do not rely on the extra non-Liouville condition for $\gamma,$ we see that it is again enough only to estimate the following sum for a $\kappa\in (0,1/3),$
		\begin{align*}
		\frac{\psi(q)}{q}\sum_{r|q} r\sum^{k\leq \kappa\log q}_{k< 2\log r}\sum_{q'\in D_{k,r}}\chi_{B(0,\Delta(q,q')/r)}(\{\gamma (q'-q)/r\}).\tag{SUM}
		\end{align*}
		We recall that for each integer $k\geq 0,$ 
		\[
		D_{k,r}=\{1\leq q'<q: \gcd(q',q)=r, q'/q\in [2^{-k-1},2^{-k})\}.
		\]
		We have
		\[
		D_{k,r}\subset \{1\leq q'<q: r|q', q'/q\in [2^{-k-1},2^{-k})\}.
		\]
		The set on the RHS in above is an arithmetic progression with gap $r$ of length at most
		\[
		\left[\frac{q}{r2^k}\right]+1.
		\]
		As we have for a constant $C>0,$
		\[
		\Delta(q,q')/r\leq \frac{2q}{q' (\log\log q')^2 r}\leq 2 \times 2^{k}\frac{C}{(\log\log q)^2 r}
		\]
		for $q'\in D_{k,r}$ with $k\leq \min\{\kappa\log q,2\log r\}.$ First, suppose that 
		\begin{align*}
		r>2^{k+2}\frac{C}{(\log\log q)^2}\times Q.\tag{B1}
		\end{align*}
		Then we see that
		\[
		\Delta(q',q)/r<\frac{1}{2Q}.
		\]
		In this case, we claim that
		\[
		\sum_{q'\in D_{k,r}}\chi_{B(0,\Delta(q,q')/r)}(\{\gamma (q'-q)/r\})\leq \left[\frac{q/r}{Q}\right].
		\]
		To see how the RHS in above follows, recall that we need to consider $q'$ being multiples of $r.$ Then we see that under the condition (B1),
		\begin{align*}
		&\sum_{q'\in D_{k,r}}\chi_{B(0,\Delta(q,q')/r)}(\{\gamma (q'-q)/r\}&\leq\sum_{r|q', q'\in [1,q)}\chi_{B(0,1/2Q)}(\{\gamma (q'-q)/r\})\\
		&=\sum_{m\in (0,q/r)}\chi_{B(0,1/2Q)}(\{\gamma m\}).&
		\end{align*}
		By what we have discussed in above, $\chi_{B(0,1/2Q)}(\{\gamma m\})$ can be non-zero only when
		$
		Q|m.
		$
		This finishes the proof of the claim. Thus for constants $C'',C'''>0,$ the contribution to (SUM) in this case is
		\begin{align*}
		&\leq C''\frac{\psi(q)}{q}\sum_{r|q}r \sum_{2^k\leq \min\{r^2,q^{1/3}\}} \frac{q}{rQ}&\\
		&\leq C''' \psi(q)\frac{\log q}{Q}d(q).&\tag{I}
		\end{align*}
		As we have $q\leq Q^{\sigma(Q)/2},$ we see that $Q\geq q^{2/\sigma(Q)}.$ Since $d(q)\log q=o(q^{2/\sigma(Q)})$, we see that (I) can be bounded by
		\[
		\psi(q)o(1).
		\]
Next, we consider the case when
		\begin{align*}
		r\leq 2^{k+2}\frac{C}{(\log\log q)^2}\times Q.\tag{B2}
		\end{align*}
		In this case we have
		\[
		2\times 2^{k}\frac{C}{(\log\log q)^2 r}\geq \frac{1}{2Q}.
		\]
		For convenience, we write
		\[
		\min\left\{1,2\times 2^{k}\frac{C}{(\log\log q)^2 r}\right\}=T_{k,r}.
		\]
		Next, observe that under the condition (B2), we have
		\[
		\frac{q}{2^k r}\geq \frac{q(\log\log q)^2}{2^{2k+2}CQ}.
		\]
		Since $2^{2k}\leq q^{2/3},$ we see that (recall that $q\geq Q^7$)
		\begin{align*}
		\frac{q}{2^k r}\geq \frac{(\log\log q)^2}{4C} \frac{q^{1/3}}{Q}\geq  \frac{(\log\log q)^2}{4C} Q^{2}.\tag{@}
		\end{align*}
		This is much larger than $Q$ if $Q$ is large enough. Then we see that for consants $C'''',C'''''>0$
		\begin{align*}
		&\sum_{q'\in D_{k,r}}\chi_{B(0,\Delta(q,q')/r)}(\{\gamma (q'-q)/r\}&\leq\sum_{r|q', q'\in [q/2^{k+1},q/2^k]}\chi_{B(0,T_{k,r})}(\{\gamma (q'-q)/r\})\\
		&\leq C''''\left(\left[\frac{T_{k,r}}{1/(2Q)}\right]+1\right)\left(\left[\frac{\frac{q}{2^k r}+1}{Q}\right]+1\right)&\tag{@@}\\
		&\leq C''''' \frac{T_{k,r}}{1/Q}\frac{q/(2^k r)}{Q}&\\
		&\leq q\frac{2CC'''''}{(\log\log q)^2}\frac{1}{r^2}.&
		\end{align*}
		Thus the contribution to (SUM) is at most
		\[
		\frac{4CC'''''}{(\log\log q)^2}\psi(q)\sum_{r|q} \frac{\log r}{r}.
		\]
		From here, together with the estimate (I) and reindexing the constants, we see that the lemma is proved.
	\end{proof}

	\subsection{Proofs of the main results}
	\begin{proof}[Proof of Theorem \ref{Main1}]
		Define the arithmetic function $F$ by
		\[
		F(q)=\sum_{r|q} \frac{\log r}{r}.
		\]
		Let $Q$ be a large integer and let $K_Q$ be a positive integer whose value depends on $Q.$ We see that
		\begin{eqnarray*}
			\sum_{q=1}^Q F^{K_Q}(q)&=&\sum_{q=1}^Q \sum_{r_1,r_2,\dots,r_{K_Q}|q} \prod_{i=1}^{K_Q} \frac{\log r_i}{r_i}\\
			&=& \sum_{r_1,r_2,\dots,r_{K_Q}\leq Q}  \prod_{i=1}^{K_Q} \frac{\log r_i}{r_i}\sum_{q: [r_1,\dots,r_{K_Q}]| q} 1\\
			&\leq& \sum_{r_1,r_2,\dots,r_{K_Q}\leq Q}  \prod_{i=1}^{K_Q} \frac{\log r_i}{r_i} \frac{Q}{[r_1,\dots,r_{K_Q}]}.
		\end{eqnarray*}
		As $[r_1,\dots,r_{K_Q}]\geq (r_1r_2\dots r_{K_Q})^{1/K_Q},$ we have
		\[
		\sum_{q=1}^Q F^{K_Q}(q)\leq Q\left(\sum_{1\leq r\leq Q} \frac{\log r}{r^{1+K^{-1}_Q}}\right)^{K_Q}.
		\]
		Observe that there is a constant $C$ with
		\[
		\sum_{1\leq r\leq Q} \frac{\log r}{r^{1+K^{-1}_Q}}\leq \sum_{r\geq 1}  \frac{\log r}{r^{1+K^{-1}_Q}}=-\zeta'(1+K^{-1}_Q)\leq C K^2_Q.
		\]
		Thus we see that
		\[
		\sum_{q=1}^Q F^{K_Q}(q)\leq Q(CK^2_Q)^{K_Q}.
		\]
		Then we see that
		\[
		\#\{q\leq Q: F(q)> 2CK^2_Q\}\leq Q \frac{1}{2^{K_Q}}.
		\]
		We choose $K_Q=2\log\log Q$ and this makes
		\[
		\#\{q\leq Q: F(q)> 4C\log\log^2 Q\}\leq \frac{Q}{\log^2 Q}.
		\]
		Since $\psi(q)=O(q^{-1}(\log\log q)^{-2})$ we see that there is a constant $C'>0$ such that
		\[
		\sum_{q: F(q)>4C \log\log^2 q} \psi(q)=\sum_{k\geq 1} \sum_{q: F(q)>4C \log\log^2 q, 2^{k-1} \leq q\leq 2^k} \psi(q)\leq  C'\sum_{k\geq 1}\frac{2^k}{2^{k-1}} \frac{1}{k^2}<\infty.
		\]
		Thus, we can assume that $\psi$ is supported on where $F(q)<4C \log\log^2 q.$ From here and Lemmas \ref{master}, \ref{Sum Estimate}, we see that
		\begin{eqnarray*}
			\sum_{1\leq q'<q} |A_q\cap A_{q'}|&\leq& 4\left(1+\frac{C_0}{2H}\right) \sum_{1\leq q'\leq q}\psi(q)\psi(q')+C''\psi(q)+C' \psi(q)\frac{F(q)}{\log\log^2 q}\\&\leq& 4\left(1+\frac{C_0}{2H}\right) \sum_{1\leq q'\leq q}\psi(q)\psi(q')+C''\psi(q)+4C'C \psi(q).
		\end{eqnarray*}
		From above and the assumption $\sum_q \psi(q)=\infty$ we see that
		\[
		\frac{(\sum_{q\leq Q} |A_q|)^2}{\sum_{q,q'\leq Q} |A_q\cap A_{q'}|}\geq\frac{(\sum_{q=1}^Q 2\psi(q))^2}{4\left(1+\frac{C_0}{2H}\right) (\sum_{q=1}^Q\psi(q))^2+O(\sum_{q=1}^Q \psi(q))}\geq \frac{1}{1+\frac{C_0}{2H}+o(1)}.
		\]
		By Lemma \ref{Borel}, we see that $|W(\psi,\gamma)|=|\limsup_{q\to\infty} A_q|\geq (1+C_0/(2H))^{-1}.$ As $H$ can be chosen to be arbitrarily large, we see that
		\[
		|W(\psi,\gamma)|=1.
		\]
	\end{proof}

	\begin{proof}[Revisiting Sz\"{u}sz's theorem]
		As $\psi$ is non-increasing, if for an integer $q\geq 100$ we have $\psi(q)\geq 1/q,$ then we have
		\[
		\psi(q')\geq \frac{1}{2q'}
		\]
		for all $q'\in [q/2,q].$ In this case, our strategy is to shrink $\psi(q')$ to $1/(2q')$ for $q'\in [q/2,q].$ More precisely, we find $q_1$ the first integer $\geq 100$ with $\psi(q_1)\geq 1/q_1.$ Then we shrink $\psi$ at $[q_1/2,q_1].$ Next, we find $q_2,$ the first integer $>q_1$ with $\psi(q_2)\geq 1/q_2.$ We then shrink $\psi$ at $[\max\{q_1+1,q_2/2\},q_2].$ In the end, we obtain a new approximation function $\psi'\leq \psi$ such that $\psi'(q)=O(1/q).$
		
		Suppose that the above shrinking procedure was performed infinitely many times. In particular, we can find numbers $s_1<s_2<\dots$ with $2s_i<s_{i+1}$  and $\psi(s_i)>1/s_i$ for $i\geq 1.$ Then the modified approximation function $\psi'$ satisfies
		\[
		\psi'(q')\geq \frac{1}{2s_i}
		\]
		for $q'\in [s_i/2,s_i].$ Thus $\sum_{q'\in [s_i/2,s_i]}\psi'(q')\geq 1/4.$ From here we see that
		\[
		\sum_{q} \psi'(q)=\infty.
		\]
		
		Now we want to sieve out the support of $\psi'$ further. Recall the function $F$ used before in the proofs of Theorems \ref{Main1},
		\[
		F(q)=\sum_{r|q} \frac{\log r}{r}.
		\]	
		We have for a constant $C>0$ that
		\[
		\#\{q\leq Q: F(q)>2CK^2_Q\}\leq Q\frac{1}{2^{K_Q}}
		\]
		holds for each $Q\geq 100$ and $K_Q\geq 100.$ Now we simply choose $K_Q$ to be a fixed large number $K$ and we restrict $\psi$ to
		\[
		G_K=\{q\leq Q: F(q)\leq 2CK^2\}.
		\]
		This set has natural lower density at least $1-1/2^K.$ Moreover, for each integer $q>100$, the set $G^c_K\cap [q/2,q]$ contains at most $q/c_K$ elements, where $c_K>1$ and it can be made to be arbitrarily large by choosing $K$ to be sufficiently large. We can restrict $\psi'$ on $G_K$ and we denote this new approximation function as $\psi'_K.$ Observe that $F(q)$ is bounded for $q\in G_K$. Instead of Lemma \ref{Sum Estimate}, we can use Remark \ref{Remark} in the proof of Theorem \ref{Main1}. 
		
		Notice that we have not used any Liouville conditions for $\gamma.$ We now consider $W(\psi'_K,\gamma).$ Following the arguments in the proofs of Theorem \ref{Main1}, we see that
		\[
		|W(\psi,\gamma)|=1.
		\]

		\begin{comment}
		Let $M>1$ be an integer. We choose $K$ to be large enough such that $M_K=M\mathbb{N}\cap G_K$ is a large set. More precisely, we need $M_K\cap [q/2,q]$ to contain at least $\epsilon_K q$ elements for all sufficiently large $q$. This can be achieved by choosing $K$ to be sufficiently large. After fixing the value of $K$ (according to $M$), we now consider $W(\psi'_K,\gamma).$ Observe that the construction in above ensures that \[\sum_{q\in M_K} \psi'_K(q)=\infty.\]
		\end{comment}

		\begin{comment}
		we see that for each interval $I\in\mathcal{I}_M,$
		\[
		|I\cap W(\psi,\gamma)|\geq |I\cap W(\psi',\gamma)|\geq |I\cap W(\psi'_K,\gamma)|\geq \frac{1}{4}.
		\]
		As $M$ can be chosen to be arbitrarily large, we conclude that
		\[
		|W(\psi,\gamma)|=1.
		\]
		\end{comment}
		Now, we need to deal with the case when the shrinking procedure above cannot be performed infinitely often. This implies that $\psi(q)<1/q$ except for at most finitely many values of $q.$ Thus in this case, we have $\psi(q)=O(1/q).$ Again, we choose $K$ such that $G_K$ has positive lower density. This can be done by choosing $K$ to be sufficiently large. As $\psi$ is non-increasing, we see that
		\[
		\sum_{q\in G_K}\psi(q)=\infty.
		\]
		Consider the restricted approximation function $\psi_K.$ As in above we see that  $|W(\psi,\gamma)|=1.$ From here the proof is finished.
	\end{proof}
	
	As a direct consequence of Theorem \ref{Main1}, we now illustrate the following.
	\begin{proof}[Proof of Corollary \ref{MUL1}]
		Consider the following set
		\[
		B=\{q\in\mathbb{N}: \|q\beta-\gamma_2\|\geq (\log q)^{-1}\}.
		\]
		Then on this set $B$, we have
		\[
		\psi'(q)=\frac{\psi(q)}{\|q\beta-\gamma_2\|}=O((q(\log\log q)^2)^{-1}).
		\]
		We extend $\psi'$ by setting $\psi'(q)=0$ whenever $q\notin B.$ Then we have
		\[
		W(\psi',\gamma_1)\subset W(\psi,\beta,\gamma_1,\gamma_2).
		\]	
		Since we have
		\[
		\sum_{q\in B} \psi'(q)=\infty,
		\]
		we can use Theorem \ref{Main1} to conclude the result.
	\end{proof}
	We now provide examples such that the condition (D) in the statement of Theorem \ref{MUL1} is satisfied. First, we want to analyse the set $B$ constructed in the previous proof. Let $\beta$ be an irrational algebraic number and $\gamma$ be a real number. We want to understand the set
	\[
	B=\{q\in\mathbb{N}: \|q\beta-\gamma_2\|\geq (\log q)^{-1}\}.
	\]
	To do this, let $k,l$ be natural numbers and we consider
	\[
	B_{k,l}=\{q\in [2^k,2^{k+1}]: \|q\beta-\gamma_2\|\in [2^l/k, 2^{l+1}/k]\}.
	\]
	We need to estimate from below the cardinality of $B_{k,l}.$ The interesting case would be $2^{l}\leq k.$ As $\beta$ is algebraic, it is not Liouville. We can use Lemma \ref{Discrepancy}. As a result, we see that there are numbers $c,M>0$ such that for all $k\geq M$ and $l\leq \log k$ we have
	\[
	\#B_{k,l}\geq c 2^k \frac{2^{l}}{k}.
	\]
	We let \footnote{This choice of $\psi$ happens to be monotonic. However, this is not essential in the argument.}
	\[
	\psi(q)=\frac{1}{q \log q(\log\log q)^2}.
	\]
	Then we see that for $k\geq M, l\leq \log k,$
	\[
	\sum_{q\in B_{k,l}} \frac{\psi(q)}{\|q\beta-\gamma_2\|}\geq c 2^k \frac{2^l}{k} \frac{1}{2^{k+1}(k+1)(\log (k+1))^2} \frac{k}{2^{l+1}}=\frac{c}{4}\frac{1}{(k+1)(\log (k+1))^2}.
	\]
	Then we have
	\[
	\sum_{q\in B\cap [2^k,2^{k+1}]} \frac{\psi(q)}{\|q\beta-\gamma_2\|}\geq \sum^{l\leq \log k}_{l=0} \frac{c}{4}\frac{1}{(k+1)(\log (k+1))^2}\geq \frac{c}{8} \frac{1}{k\log k}
	\]
	for all sufficiently large $k.$ Thus we conclude that
	\[
	\sum_{q\in B} \frac{\psi(q)}{\|q\beta-\gamma_2\|}=\infty.
	\]
	\section{Higher dimensional approximations}
	We will prove Theorem \ref{HIGH}.  Let $H,C_0$ be as in Lemma \ref{master}. Let $\psi$ be an approximation function with the required divergence condition. For each $q\geq 2,$ we consider
	\[
	B_q=\prod_{i=1}^k A^{\psi,\gamma_i}_q.
	\]
	We want to study the set $\limsup_{q\to\infty} B_q.$ Now, observe that the Lebesgue measure of $B_q$ is simply
	\[
	(2\psi(q))^k
	\]
	if $\psi(q)\leq 1/2.$ Otherwise, the Lebesgue measure is $1.$ As before, we will always assume $\psi(q)<1/2,$ or else the result follows trivially. We can use Lemma \ref{master} for each component. Write for integers $i\in\{1,\dots,k\}$ and $q,q'\geq 2,$
	\begin{align*}
	L_i(q,q')=2(2H+1)\min\{\psi(q)/q,\psi(q')/q'\}\gcd(q,q')\chi_{B(0,\Delta(q,q')/\gcd(q,q'))}(\{\gamma_i(q'-q)/\gcd(q,q')\})
	\end{align*}
	if $\Delta(q,q')\leq H\gcd(q,q').$ Otherwise, we define.
	\[
	L_i(q,q')=4(1+C_0/(2H))\psi(q)\psi(q').
	\]
	Then we see that
	\[
	|B_q\cap B_{q'}|\leq \prod_{i=1}^k L_i(q,q').
	\]
	When $\Delta(q,q')\leq H\gcd(q,q'),$ we shall simply estimate $L_i(q,q')$ from above by
	\[
	L_i(q,q')\leq 2(2H+1)\min\left\{\frac{\psi(q)}{q},\frac{\psi(q')}{q'}\right\}\gcd(q,q')
	\]
	From here we removed the dependence of the inhomogeneous shifts $\gamma_1,\dots,\gamma_k.$ We now estimate $\sum_{1\leq q'\leq q} |B_q\cap B_{q'}|.$ We split the sum according to whether $\Delta(q,q')$ less or larger than $H\gcd(q,q').$ For the latter part, the upper bound is
	\[
	\sum_{1\leq q'\leq q, \Delta(q,q')\geq H\gcd(q,q')}  |B_q\cap B_{q'}|\leq \sum_{1\leq q'\leq q} \sum_{1\leq q'\leq q}\prod_{i=1}^k 4^k(1+C_0/(2H))^k\psi^k(q)\psi^k(q').\tag{I}
	\]
	For the former part, we have
	\begin{align*}
	\sum_{1\leq q'\leq q, \Delta(q,q')\leq H\gcd(q,q')}  |B_q\cap B_{q'}|&\leq& \sum_{1\leq q'\leq q} \left(2(2H+1)\min\left\{\frac{\psi(q)}{q},\frac{\psi(q')}{q'}\right\}\gcd(q,q')\right)^k\\
	&\leq& 2^k(2H+1)^k \sum_{1\leq q'\leq q} \frac{\psi(q)^k}{q^k} \gcd(q,q')^k\\ 
	&=& 2^k(2H+1)^k \frac{\psi(q)^k}{q^k}\sum_{r|q} r^k \sum_{1\leq q'\leq q, \gcd(q,q')=r} 1.\\
	&=& 2^k(2H+1)^k \frac{\psi(q)^k}{q^k}\sum_{r|q} r^k \phi(q/r). \tag{II}
	\end{align*}
	Suppose that $k\geq 3.$ We then use the trivial bound $\phi(q/r)\leq q/r$ to write (II) further as
	\[
	\leq 2^k (2H+1)^k\frac{\psi(q)^k}{q^k}\sum_{r|q} r^k \frac{q}{r}= 2^k(2H+1)^k \psi(q)^k\sum_{r|q} \left(\frac{r}{q}\right)^{k-1}\leq 2^k(2H+1)^k \psi(q)^k \zeta(k-1). \tag{II'}
	\]
	From here we see that for $k\geq 3,$
	\[
	\sum_{1\leq q'\leq q} |B_q\cap B_{q'}|\leq 4^k(1+C_0/(2H))^k \psi(q)^k (\sum_{1\leq q'\leq q} \psi(q')^k)+2^k(2H+1)^k\zeta(k-1) \psi(q)^k.
	\]
	From here, we can use Lemma \ref{Borel} to conclude that
	\[
	|\limsup_{q\to\infty} B_q|\geq \frac{1}{(1+C_0/(2H))^k}.
	\]
	This proves the result for $k\geq 3$ as $H$ can be chosen to be arbitrarily large.
	
	When $k=2,$ Estimate (II') fails as $\zeta(1)$ is not defined. We need to argue differently. We have the following estimate from (II) that for a constant $C>1,$
	\[
	2^2(2H+1)^2 \psi(q)^2 \sum_{r|q}\frac{1}{r}\leq 2^2 C \psi(q)^2(2H+1)^2 \frac{q}{\phi(q)}.
	\]
	Here we have used the fact that for a constant $C>1,$ for all $q\geq 1,$
	\[
	C^{-1}\frac{q}{\phi(q)}\leq \sum_{r|q} \frac{1}{r}\leq C\frac{q}{\phi(q)}.
	\]
	For each $l\geq 0,$ let $D_l$ be the set
	\[
	\{q: q/\phi(q)\in [2^l,2^{l+1}]\}.
	\]
	Denote the sum
	\[
	a_l=\sum_{q\in D_l} \psi(q)^2 \left(\frac{\phi(q)}{q}\right)^2.
	\]
	Suppose that $a_l=\infty$ for an integer $l\geq 0.$ In this case, we can just restrict $\psi$ to the set $D_l.$ In addition, we have  $q/\phi({q})\leq 2^{l+1}$  on $D_l.$ Then, by using the same argument as in the $k\geq 3$ case we see that
	\[
	|\limsup_{q\to\infty} B_q|\geq 1/(1+C_0/(2H))^2.
	\]
	Thus we assume that
	\[
	a_l<\infty
	\]
	for all $l\geq 0.$ The divergent condition for $\psi$ forces $\sum_l a_l=\infty.$ Now let $l$ be any integer. By Lemma \ref{EC}, we see that as long as at least one $B_q$ with $q\in D_l$ has positive measure, we have for all large enough $Q,$
	\[
	\mu_Q=|\cup_{q\leq Q, q\in D_l} B_q|\geq \frac{\left(\sum_{q=1}^{Q}{|B_q|}\right)^2}{\sum_{q,q'=1}^{Q}{|B_q \cap B_{q'}|}}.
	\]
	We see that
	\[
	\liminf_{q\to\infty} \mu_Q\geq \frac{1}{(1+C_0/(2H))^2+2^2(2H+1)^2 C 2^{l+1} \frac{1}{a_l \times 2^{2l}}}.\tag{III}
	\]
	Warning, this does not imply that $|\limsup_{q\to\infty} B_q|>0$! However, as $\sum_l a_l=\infty,$ there must exist arbitrarily large integers $l$ such that
	\[
	2^{l}a_l\geq 1000\times 2^3\times C\times (2H+1)^2.
	\]
	From here we see that there are arbitrarily large integers $Q$ such that
	\[
	\mu_Q\geq \frac{1-1000^{-1}}{(1+C_0/(2H))^2+1000^{-1}}.
	\]
	We fix a pair of $l,Q$ as in above and write $H_{l}=\cup_{q\leq Q, q\in D_l} B_q.$ We do not explicitly write $Q$ here. As a consequence, we can find sets $H_{l}$ with arbitrarily large $l$ (and corresponding $Q$) such that
	\[
	|H_l|\geq \frac{1-1000^{-1}}{(1+C_0/(2H))^2+1000^{-1}}.
	\]
	We can rename $H_l$ such that the index $l$ runs over $1,\dots,\infty.$ Notice that $F_n=\cup_{k\geq n} H_k$ is a decreasing sequence of Lebesgue measurable sets. This implies that (by the continuity of Lebesgue measure)
	\[
	|\limsup_{l\to\infty} H_l|=\left|\bigcap_{n\geq 1}\left(\bigcup_{k\geq n} H_k\right)\right|=|\cap_{n\geq 1} F_n|=\lim_{n\to\infty} |F_n|\geq \frac{1-1000^{-1}}{(1+C_0/(2H))^2+1000^{-1}}.
	\]
	Suppose that $x\in \limsup_{l\to\infty} H_l,$ then $x\in H_l$ for infinitely many $l.$ This implies that $x\in B_q$ for at least one $q\in D_l$ for infinitely many $l.$ This implies that $x\in B_q$ for infinitely many different $q.$ This finishes the proof for $k=2$(notice that the number $1000$ as well as $H$ in above can be replaced by any large numbers).
	
	Now let us assume that $k=1.$ In this case, (II) can be further bounded by
	\[
	\leq 2\psi(q) d(q).
	\]
	As we required $\psi(q)/(d(q))^{1+\epsilon}=\infty,$ we can perform the argument as in the $k=2$ case. We just discuss values of $d(q)$ instead of $q/\phi(q)$. The extra $+\epsilon$ on the exponent will help us to find an estimate like (III) where the coefficient $a_l$ is $2^{(1+\epsilon)l}.$ From here the rest of the arguments can be performed without essential changes. This concludes the proof.
	\subsection*{Acknowledgements.}
	HY was financially supported by the University of Cambridge and the Corpus Christi College, Cambridge. HY has received funding from the European Research Council (ERC) under the European Union’s Horizon 2020 research and innovation programme (grant agreement No. 803711). HY thanks S. Chow for helpful comments.

	\Addresses
	

\begin{thebibliography}{999}
		
		
		\bibitem{ALMTZ} C. Aistleitner, T. Lachmann, M. Munsch, N. Technau and A. Zafeiropoulos, \emph{The Duffin-Schaeffer conjecture with extra divergence}, Adv. Math. \textbf{356}(7), (2019).
		
		
		
		\bibitem{BDV ref} V. Beresnevich, D. Dickinson, S. Velani,
		\emph{Measure theoretic laws for lim sup sets}, Mem. Amer. Math. Soc. \textbf{179}(846), (2006).
		
		\bibitem{BHHV} V. Beresnevich, G. Harman, A. Haynes and S. Velani, \emph{The Duffin-Schaeffer conjecture with extra divergence II}, Math. Z. \textbf{275}(1), (2013), 127–133.
		
		\bibitem{BHV}V. Beresnevich, A. Haynes and S. Velani, \emph{Sums of reciprocals of fractional parts
			and multiplicative Diophantine approximation}, to appear in Mem. Amer. Math. Soc., preprint, arXiv:1511.06862.
		
		\bibitem{C50} J. Cassels, \emph{Some metrical theorems in Diophantine approximation}, I, Proc. Cambridge Philos. Soc. \textbf{46}, (1950), 209–218.
		
		\bibitem{C18} S. Chow, \emph{Bohr sets and multiplicative Diophantine approximation},  Duke Math. J. \textbf{167}(9), (2018), 1623-1642.
		
		
		\bibitem{EC} K. L. Chung, P. Erd\H{o}s, \emph{On the application of the Borel-Cantelli lemma}, Trans. Amer. Math. Soc. 72 (1952), 179--186.
		
		\bibitem{DS} R. Duffin and A. Schaeffer, \emph{ Khintchine’s problem in metric Diophantine approximation}, Duke Math. J. \textbf{8} ,(1941), 243–255.
		
		\bibitem{DT97} M. Drmota and R. Tichy, \emph{ Sequences, Discrepancies and Applications}, Lecture Notes in Mathematics,
		Springer-Verlag Berlin Heidelberg,(1997).
		
		
		\bibitem{Gallagher} P. Gallagher, \emph{Approximation by reduced fractions}, J. Math. Soc. Japan \textbf{13}(4), (1961), 342–345.
		
		\bibitem{Gallagher65} P. Gallagher, \emph{Metric simultaneous diophantine approximation (II)}, Mathematika \textbf{12}(2), (1965), 123-127.
		
		\bibitem{Harman} G. Harman, \emph{Metric number theory}, Clarendon Press, Oxford, (1998).
		
		
		\bibitem{Khintchine} A. Khintchine, \emph{Einige S\"{a}tze \"{u}ber Kettenbr\"{u}cke, mit Anwendungen auf die Theorie der Diophantischen Approximationen}, Math. Ann. \textbf{92}(1-2), (1924), 115-125.
		
		\bibitem{KM2019} D. Koukoulopoulos and J. Maynard, \emph{On the Duffin--Schaeffer conjecture}, arXiv:1907.04593.
		
		\bibitem{PV} A. Pollington and R. Vaughan, \emph{The k-dimensional Duffin and Schaeffer conjecture}, J. Th\'{e}or. Nombres Bordeaux \textbf{1}(1), (1989), 81-88.
		
		
		\bibitem{Ramirez} F. Ram\'{i}rez, \emph{Counterexamples, covering systems, and zero-one laws for inhomogeneous approximation}, Int. J. Number Theory, \textbf{13}(3), (2017), 633-654. 
		
		
		
		\bibitem{Szusz} P. Sz\"{u}sz, \emph{\"{U}ber die metrische Theorie der Diophantischen Approximation}, Acta. Math. Sci. Hungar. \textbf{9}, (1958) ,177-193.
		
		\bibitem{Vaaler} J. Vaaler, \emph{On the metric theory of Diophantine approximation}, Pacific J. Math. \textbf{76}(2), (1978), 527–539.
		
		
		\bibitem{Yu} H. Yu, \emph{A Fourier-analytic approach to inhomogeneous Diophantine approximation}, Acta Arith. \textbf{190}, (2019), 263-292.
		
		
		
	\end{thebibliography}
\end{document}